\documentclass[11pt]{article}
\usepackage{amsmath}
\usepackage{amssymb}
\usepackage{amsthm}
\usepackage{enumerate}
\usepackage{graphicx,color}

\parskip \medskipamount
\newtheorem{theorem}{Theorem}[section]
\newtheorem{lemma}[theorem]{Lemma}
\newtheorem{corollary}[theorem]{Corollary}

\newtheorem{remark}[theorem]{Remark}

\newcommand{\ind}[1]{\mathbf{1}\left(#1\right)}
\newcommand{\pr}[1]{\operatorname{\mathbf{P}}\left(#1\right)}

\newcommand{\prcond}[2]{\operatorname{\mathbf{P}}\left(#1\;\middle\vert\;#2\right)}

\newcommand{\E}[1]{\operatorname{\mathbf{E}}\left[#1\right]}

\newcommand{\vol}[1]{\operatorname{vol}\left(#1\right)}

\newcommand{\origin}{0}
\newcommand{\critical}{\mathrm{c}}

\newcommand{\tperc}{T_\mathrm{perc}}
\newcommand{\tnonperc}{T_\mathrm{nonperc}}
\newcommand{\tdet}{T_\mathrm{det}}
\newcommand{\tisol}{T_\mathrm{isol}}

\newcommand{\1}{{\text{\Large $\mathfrak 1$}}}

\newcommand{\IGNORE}[1]{}

\def\rset{{\mathbb{R}}}
\def\R{{\rset}}
\def\zset{{\mathbb{Z}}}

\def\N{\mathbb{N}}
\def\D{\mathcal{D}}
\def\til{\widetilde}

\begin{document}

\title{The Isolation Time of Poisson Brownian Motions}
\author{Yuval~Peres\thanks{Microsoft Research, Redmond WA, U.S.A.\ \ Email:
       \hbox{peres@microsoft.com}.}
\and
       Perla~Sousi\thanks{University of Cambridge, Cambridge, U.K.\ \ Email:
       \hbox{p.sousi@statslab.cam.ac.uk}.}
\and
       Alexandre~Stauffer\thanks{Microsoft Research, Redmond WA, U.S.A.\ \ Email:
       \hbox{alstauff@microsoft.com}.}
}

\maketitle
\thispagestyle{empty}

\begin{abstract}
Let the nodes of a Poisson point process move independently in $\R^d$ according to Brownian motions.
We study the isolation time for a target particle that is placed at the origin,
namely how long it takes until there is no node of the Poisson point process within distance $r$ of it.
In the case when the target particle does not move, 
we obtain asymptotics for the tail \mbox{probability} which are tight up to constants in the exponent in dimension $d\geq 3$
and tight up to logarithmic factors in the exponent for dimensions $d=1,2$. In the case when the target particle 
is allowed to move independently of the Poisson point process, we show that the best strategy for the target to avoid isolation is to stay put.
\newline
MSC 2010 \emph{subject classifications.}
Primary 60G55; 
Secondary 60J65. 
\end{abstract}

\setcounter{footnote}{0}
\setcounter{page}{1}
\section{Introduction}
Let $\Pi_0 = \{ X_i\}$ be a Poisson point process over $\mathbb{R}^d$ with intensity $\lambda>0$. To avoid ambiguity, we refer to the points of
$\Pi_0$ as nodes.
For each $s> 0$, let $\Pi_s$ be obtained by letting the nodes of $\Pi_0$ move according to  independent Brownian motions.
More formally, for each $X_i \in \Pi_0$, let $(\xi_i(t))_t$ be a Brownian motion starting at the origin of $\mathbb{R}^d$, independent over different $i$. We define
$$
   \Pi_s = \bigcup_{i} \{X_i+\xi_i(s)\}.
$$
It follows by standard arguments (see e.g.~\cite{vandenberg}) that, for any fixed $s\geq 0$, the process $\Pi_s$ is a Poisson point process of intensity $\lambda$.
Henceforth we consider $\lambda$ and $r$ to be fixed constants and omit dependencies on these quantities from the notation.


We add a target particle at the origin of $\mathbb{R}^d$ at time $0$ and consider the case where this particle does not move.
We define the \emph{isolation time} $\tisol$ as the first time $t$ at which all nodes of $\Pi_t$ have distance at least $r$ from the target particle.
More formally,
$$
   \tisol = \inf\Big\{t\geq 0 \colon 0 \not \in \bigcup_{i}B(X_i+\xi_i(t),r)\Big\},
$$
where $B(x,r)$ denotes the ball of radius $r$ centered at $x$.

In this paper we derive bounds for $\pr{\tisol >t}$ that are tight up to logarithmic factors in the exponent.
In order to simplify the statement of our theorems, we define the function $\Psi_d(t)$ as
\begin{equation}
   \Psi_d(t) = \left\{\begin{array}{rl}
      \sqrt{t}, &\text{ for $d=1$}\\
      \log t, &\text{ for $d=2$}\\
      1, &\text{ for $d\geq 3$}.
   \end{array}\right.
   \label{eq:psi}
\end{equation}
Then, Theorem~\ref{thm:upper},
whose proof is given in Section~\ref{sec:upper},
establishes an upper bound for the tail of $\tisol$.
\begin{theorem}\label{thm:upper}
   For all $d\geq 1$ and any $\lambda, r>0$, there exist $t_0>0$ and a positive constant $c$ such that
   $$
      \pr{\tisol > t} \leq \exp\left(-c \frac{t}{\Psi_d(t)}\right),
   $$
   for all $t\geq t_0$.
\end{theorem}

It is easy to see that there is a positive constant $c$ so that the following lower bound holds in all dimensions~$d\geq 1$:
\begin{equation}
   \pr{\tisol >t} \geq \exp\left(- c t\right).
   \label{eq:lowerbound}
\end{equation}
This is true since, with constant probability, there is a node within distance $r/2$ of the origin at time $0$ and the probability that, from time $0$ to $t$,
this node never leaves a ball of radius
$r/2$ centered at its initial position is $e^{-\Theta(t)}$ (see for instance \cite{CTaylor});
this implies that this node is within distance $r$ of the origin throughout the time interval $[0,t]$ with probability $e^{-\Theta(t)}$.
Here and in the rest of the paper, $\Theta(t)$ denotes any function which is bounded above and below by constant multiples of $t$.
Comparing Theorem~\ref{thm:upper} with~\eqref{eq:lowerbound},
we see that the exponent in Theorem~\ref{thm:upper} is tight up to constants for $d\geq 3$ and tight up to logarithmic factors for $d=2$.
For $d=1$, the lower bound in~\eqref{eq:lowerbound} is far from the upper bound in Theorem~\ref{thm:upper}. We
obtain a better lower bound in the next theorem, which we prove in Section~\ref{sec:lower}. This lower bound matches the upper bound up to logarithmic factors in the exponent.
\begin{theorem}\label{thm:lower}
   For $d=1$ and any $\lambda, r>0$, there exist $t_0>0$ and a positive constant $c$ such that
   $$
      \pr{\tisol > t} \geq \exp\left(-c \sqrt{t}\log t \log\log t\right),
   $$
   for all $t\geq t_0$.
\end{theorem}

The isolation time, as defined above, can be generalized in two distinct ways: by replacing the balls in the definition of $\tisol$ with 
more general sets of the same volume, or by allowing the target particle to move. 
The next theorem, which we prove in Section~\ref{sec:beststrategy}, 
establishes that both these generalization can only decrease the tail of the isolation time. 
In order to state the theorem, let $(D_s)_{s\geq 0}$ be a collection of closed sets in $\R^d$. 
We say that the target is detected at time $t$ if some node of the Poisson point process is in the set $D_t$ at time $t$. 
We define the isolation time in this context as 
\[
\tisol^{D} = \inf\{t\geq 0: \forall i, \ X_i + \xi_i(t) \notin D_t \}.
\]



\begin{theorem}\label{thm:beststrategy}
Let $(D_s)_s$ be a collection of closed sets in $\R^d$ that are uniformly bounded, i.e., there exists $L_t>0$ such that $\cup_{s\leq t}D_s \subseteq B(0,L_t)$. 
Then, for all $t\geq 0$, we have
\[
\pr{\tisol^D >t} \leq \pr{\tisol^B>t},
\]
where $(B_s)_s$ are closed balls in $\R^d$ centered at the origin with
$\vol{B_s} = \vol{D_s}$ for all $s$. 
\end{theorem}

The corollary below handles the case when the target moves independently of the nodes of $\Pi_0$;
it establishes that the best strategy for the target to avoid isolation is to stay put. 
This is obtained by letting $g(s)$ be the location of the target at time $s$, and setting $D_s = B(g(s),r)$ in Theorem~\ref{thm:beststrategy}. 
\begin{corollary}\label{cor:strategy}
  Let $r>0$. Let the location of the target be given by a function $g\colon \mathbb{R}_+ \to \mathbb{R}^d$ that is bounded on compact time intervals and is independent of the nodes of $\Pi_0$.
  If we define
   $$
      \tisol^g = \inf\Big\{t\geq 0 \colon g(t) \not \in \bigcup_{i}B(X_i+\xi_i(t),r)\Big\},
   $$
   then, for any $t\geq 0$, the probability $\pr{\tisol^g > t}$ is maximized when $g\equiv 0$.
\end{corollary}

The isolation time is closely related to other quantities involving Poisson Brownian motions that have been studied in the context of mobile geometric graphs. 
We discuss these connections and give some motivation in Section~\ref{sec:questions}, where we also discuss some open problems.

\section{Proof of the upper bound}\label{sec:upper}

We start with a high-level description of the proof of Theorem~\ref{thm:upper}. We fix $\lambda$ and $r$, and let $K$ be a large positive constant. We take the nodes of $\Pi_0$ and split them into $K$ independent Poisson point processes $\Phi_1,\Phi_2,\ldots,\Phi_K$ of intensity $\displaystyle \tfrac{\lambda}{K}$ each.
We consider the first Poisson point process $\Phi_1$ and look at the amount of time during the time interval $[0,t]$ that the origin has been detected by at least one node of $\Phi_1$. We show that this quantity
is at most $t/2$, with probability at least $1-e^{-c t/\Psi_d(t)}$, for some positive constant $c$. This can be achieved by setting $K$ sufficiently large.
Then, considering only the times at which the origin has not been detected by the nodes of $\Phi_1$, which we denote by $I_2$,
we show that the amount of time within $I_2$ that the origin is detected by a node of $\Phi_2$ is at most $t/4$, with probability at least $1-e^{-ct/\Psi_d(t)}$.
Then, we repeat this procedure for each Poisson point process $\Phi_j$, and,
considering only the times at which the origin has not been detected by the nodes of $\Phi_1,\Phi_2,\ldots,\Phi_{j-1}$, which we denote by $I_j$,
we show that the amount of time within $I_j$ that the nodes of $\Phi_j$ detect the origin is at most $\frac{t}{2^j}$ with probability at least $1-e^{-ct/\Psi_d(t)}$. Then,
taking the union bound over $j=1,2,\ldots,K$,
we have that, with probability at least $1-Ke^{-ct/\Psi_d(t)}$,
the amount of time the origin has been detected by at least one node of $\Pi_0$ during $[0,t]$ is at most $\frac{t}{2}+\frac{t}{4}+\frac{t}{8}+\cdots+\frac{t}{2^K} < t$.
We remark that the sets $I_2,I_3,\ldots$, will be slightly different than the definition above, but we defer the details to the formal proof, which we give below.
\begin{proof}[{\bf Proof of Theorem~\ref{thm:upper}}]
   Let $K$ be a fixed and sufficiently large integer and
   define $\Phi_1,\Phi_2,\ldots,\Phi_K$ to be independent Poisson point processes of intensity $\tfrac{\lambda}{K}$ each.
   Using the superposition property of Poisson processes, we obtain that $\cup_{j=1}^{K}\Phi_j$ is also a Poisson point process in $\R^d$ of intensity $\lambda$. Thus,
   we can couple the nodes of $\Pi_0$ with the nodes of $\Phi_1,\Phi_2,\ldots,\Phi_K$ so that $\Pi_0=\cup_{j=1}^{K}\Phi_j$.

   Denote the points of $\Phi_j$ by $\{X^{(j)}_i\}_{i=1,2,\ldots}$ and let $(\xi^{(j)}_i(s))_{s\geq 0}$ be the Brownian motion that $X_i^{(j)}$ performs, independent over different $i$ and $j$. Thus the position of the node $X_i^{(j)}$ of $\Phi_j$ at time $s$ is
   $X^{(j)}_i+\xi^{(j)}_i(s)$. We say that a node \emph{detects} the origin at time $s$ if the node is inside the ball $B(0,r)$ at time $s$.

   Now, let $I_1=[0,t]$ and
   $$
      Z_1 = \{s \in I_1: \exists X^{(1)}_i \in \Phi_1 \text{ s.t. } X^{(1)}_i+\xi^{(1)}_i(s) \in B(0,r)\}.
   $$
   In words, $Z_1$ is the set of times during the interval $[0,t]$ at which the origin is detected by at least one node of $\Phi_1$.
   Then, for $j\geq 2$, we define inductively $J_j=[0,t]\setminus (\cup_{\ell=1}^{j-1} Z_{\ell})$,
   which is the set of times at which no node of $\Phi_1\cup\Phi_2\cup\cdots\cup \Phi_{j-1}$
   detects the origin.
   Our goal is to analyze the amount of time within $J_j$ that nodes of $\Phi_j$ detect the origin.
   However, when $J_j$ turns out to be large, it will be convenient
   to consider only a subset of $J_j$ of given size. We will denote the set of times we consider by $I_j$ and,
   for any given subset $A\subset \R$,
   we define $|A|$ to be the Lebesgue measure of $A$.
   Then, if $|J_j|\leq \tfrac{t}{2^{j-1}}$, we set $I_j=J_j$;
   otherwise, we let $I_j$ be an arbitrary subset of $J_j$ such that $|I_j|=\tfrac{t}{2^{j-1}}$.
   With this, let $Z_j$ be the set of times within the set $I_j$ at which the origin is detected by at least one node of $\Phi_j$; more formally, we have
   \[
   Z_j = \{ s \in I_j\colon \exists X^{(j)}_i \in \Phi_j \text{ s.t. } X^{(j)}_i+\xi^{(j)}_i(s) \in B(0,r) \}.
   \]
   The lemma below gives a bound for the probability that $|Z_j|$ is large.
   \begin{lemma}\label{lem:iteration}
      For all dimensions $d\geq 1$, there exists a constant $c$ such that, for any $j\in\{1,2,\ldots,K\}$, we have
      $$
         \pr{|Z_j| > \frac{t}{2^j}} \leq \exp{\left(-c\frac{t}{\Psi_d(t)}\right)}.
      $$
   \end{lemma}

   We will give the proof of Lemma~\ref{lem:iteration} in a moment; first, we show how to use Lemma~\ref{lem:iteration} to complete the proof of Theorem~\ref{thm:upper}.
   Clearly, if $|Z_j|\leq \tfrac{t}{2^j}$ for all $j$, then the amount of time at which at least one node of $\Pi_0$ detects the origin is at most $\sum_{j=1}^K \tfrac{t}{2^j}<t$, which yields
   $\tisol \leq t$. Therefore, using this and the union bound, we have
   \begin{align*}
      \pr{\tisol>t}
      \leq  \pr{\bigcup_{j=1}^{K} \left\{|Z_j| > \frac{t}{2^j}\right\}}
      \leq  \sum_{j=1}^{K} \pr{|Z_j|>\frac{t}{2^j}}
      \leq  K \exp\left(-c\frac{t}{\Psi_d(t)}\right),
   \end{align*}
   which completes the proof of Theorem~\ref{thm:upper}.
\end{proof}

Before proving Lemma~\ref{lem:iteration} we introduce some notation and prove a few preliminary results.

   In what follows we fix $j \in\{1,2,\ldots,K\}$. Let
   $$
   \Phi_j' = \{X_i^{(j)} \in \Phi_j: \exists s \in [0,t] \text{ s.t. }  X_i^{(j)}+\xi_i^{(j)}(s) \in B(0,r)\};
   $$
   that is, $\Phi_j'$ is the set of nodes of $\Phi_j$ that detect the origin at some time in $[0,t]$.
   Then $\Phi_j'$ is a thinned Poisson point process with intensity given by $\displaystyle \Lambda(x)=\tfrac{\lambda}{K}\pr{x \in \cup_{s \leq t}B(\xi(s),r)}$,
   where $(\xi(s))_s$ is a standard Brownian motion.

   Let $N_j$ be a Poisson random variable of mean
   \begin{align}\label{eq:lambda}
   \E{N_j} = \Lambda(\R^d) =\frac{\lambda}{K} \E{\vol{W_0(t)}},
   \end{align}
   where
   \begin{align}\label{eq:saucisse}
   W_0(t) = \cup_{s\leq t}B(\xi(s),r)
   \end{align}
   is the Wiener sausage with radius $r$ up to time $t$.
   It is known (see for instance~\cite{BMS, Spitzer}) that, as $t \to \infty$, the expected volume of the Wiener sausage satisfies
   \begin{equation}
   \E{\vol{W_\origin(t)}} = \frac{c(d,r) t}{\Psi_d(t)}(1+o(1)),
   \label{eq:wienervol}
   \end{equation}
   for an explicit positive constant $c(d,r)$.

   For all $\ell=1,2,\ldots$, we let $X_\ell$ be i.i.d.\ random variables in $\R^d$ distributed according to $\tfrac{\Lambda(x)}{\Lambda(\R^d)}$ and $(\xi_\ell(s))_s$ be a Brownian motion
   conditioned on $X_\ell + \xi_\ell$ hitting the ball $B(0,r)$
   before time $t$, independent over different $\ell$. Finally we define
   \[
   S_\ell = \int_{I_j} \ind{X_\ell+\xi_\ell(s) \in B(0,r)} \, ds,
   \]
   i.e.\ the time in $I_j$ that $X_\ell+\xi_\ell$ spends in the ball $B(0,r)$.

   \begin{lemma}\label{lem:domination}
   We have that
   \[
   \pr{|Z_j| > \frac{t}{2^j}} \leq \pr{\sum_{\ell=1}^{N_j} S_\ell > \frac{t}{2^j}}.
   \]
   \end{lemma}

\begin{proof}[{\bf Proof}]
   Let
   $M_j=|\Phi'_j(\R^d)|$ be the total number of nodes of $\Phi_j'$. Then $M_j$ is a Poisson random variable of mean
   \begin{equation}
      \E{M_j} = \frac{\lambda}{K}\int_{\R^d} \pr{x\in \bigcup_{s \leq t} B(\xi(s),r)} \,dx
      = \frac{\lambda}{K} \E{\vol{W_0(t)}},
      \label{eq:intensity}
   \end{equation}
   where $W_0(t)$ is the Wiener sausage as defined above. Hence, $M_j$ has the same law as $N_j$.

   Let $\Phi_j' = \{X_1',\ldots, X_{M_j}' \}$. Then the positions of the nodes $X_i'$ are independent and distributed according to $\frac{\Lambda(x)}{\Lambda(\R^d)}$. For each $\ell\in\{1,2,\ldots,M_j\}$ we define
   $T_\ell$ to be the time within $I_j$ that node $X_\ell'$ spends in $B(0,r)$, i.e.
   \[
   T_\ell = \int_{I_j} \ind{X_\ell' + \xi'_\ell(s) \in B(0,r)} \, ds,
   \]
   where $\xi'_\ell$ has the distribution of a Brownian motion conditioned on $X'_\ell + \xi'_\ell$ hitting the ball $B(0,r)$ before time $t$.

   We have that $|Z_j|$ is no larger than the sum $T_1 + \ldots + T_{M_j}$, which gives that
   \begin{equation}
      \pr{|Z_j| > \frac{t}{2^j}} \leq \pr{\sum_{\ell=1}^{M_j} T_\ell > \frac{t}{2^j}}.
      \label{eq:step1}
   \end{equation}
   By the definition of $N_j$ and $S_\ell$, for all $\ell$, we deduce that $\sum_{\ell=1}^{{N}_j}S_\ell$ has the same law as $\sum_{\ell=1}^{M_j}T_\ell$ and this together
   with~\eqref{eq:step1} concludes the proof.
   \end{proof}

   \begin{remark}\label{rem:station}
   \rm{
   By standard properties of Poisson processes, the process $\{X_i^{(j)} + \xi_i^{(j)}(s) \}_i$ is a Poisson point process of intensity $\lambda$, for every $s$ (see e.g.~\cite{vandenberg}).
   Using that fact and Fubini's theorem we have
   \begin{align}\label{eq:wald2}
      \E{\sum_{\ell=1}^{M_j} T_\ell} = \int_{I_j} \E{\sum_{i}\ind{X_i^{(j)}+\xi_i^{(j)}(s) \in B(0,r)}} \, ds
      = \frac{\lambda \omega_d r^d |I_j|}{K},
   \end{align}
   where $\omega_d$ stands for the volume of the unit ball in $\R^d$.
   Also, by the equality in law mentioned in the proof of Lemma~\ref{lem:domination} and independence we get
   \begin{align}\label{eq:wald}
     \E{\sum_{\ell=1}^{M_j} T_\ell} =\E{\sum_{\ell=1}^{{N}_j} {S}_\ell} = \E{{N}_j} \E{{S}_1}.
   \end{align}
   }
   \end{remark}

   We now introduce a sequence of i.i.d.\ random variables given by
   \[
   Y_\ell = \frac{S_\ell - \E{S_\ell}}{\Psi_d(t)}, \text{ for all $\ell=1,2,\ldots$}
   \]
We emphasize that the random variables $(S_\ell)$ and $(Y_\ell)$ depend on $t$.

   \begin{lemma}\label{lem:expon}
   There exists a positive constant $\gamma$ such that
   \[
   \sup_{t\geq 0}\E{e^{\gamma Y_1}} \leq C,
   \]
   where $C$ is a positive finite constant.
   \end{lemma}

   \begin{proof}[{\bf Proof}]
   Let $\zeta$ be a Brownian motion started according to $\tfrac{\Lambda(x)}{\Lambda(\R^d)}$ and conditioned on hitting $B(0,r)$ before time $t$. Then the construction of $(S_\ell)$ gives that
    ${S}_1$  has the same law as the time that $\zeta$ spends in $B(0,r)$ before time $t$.
   Note that after hitting $\partial B(0,r)$, the process $\zeta$ evolves as an unconditioned Brownian motion.
   For any $x$, if $\xi$ is a standard Brownian motion, then the time $L_x$ in $[0,t]$ that $x+\xi$ spends in the ball $B(0,r)$ satisfies
   \begin{align}\label{eq:expect1}
   \E{L_x} = \E{\int_{0}^{t} \ind{x+\xi(s) \in B(0,r)} \, ds} \leq 1+ \int_{1}^t \int_{B(0,r)}\frac{1}{(2\pi s)^{d/2}}\,dy\,ds \leq c_1\Psi_d(t),
   \end{align}
   for some positive constant $c_1$. By rotational invariance of Brownian motion we have that
   \begin{align}\label{eq:stochdom}
    S_1 \text{ is stochastically dominated by } L_x, \text{ for any $x$ on the boundary of $B(0,r)$}.
   \end{align}
   Using this, we will show that there exists a positive constant $c_2$ such that for all $n \geq 1$
   \begin{align}\label{eq:expon}
   \pr{S_1>n c_{1}\Psi_d(t)} \leq e^{-c_2 n}.
   \end{align}
 Before showing~\eqref{eq:expon}, we explain how we use it to prove the lemma. From the definition of $Y_1$ we get that for all $n$
 \[
 \pr{Y_1 > c_1 n} \leq e^{-c_2 n}.
 \]
   This shows that this exponential tail bound is independent of $t$, and hence there exists a $\gamma>0$ such that $\sup_{t\geq 0}\E{e^{\gamma Y_1}} \leq C< \infty$.
   
   In order to show~\eqref{eq:expon}, note that, by \eqref{eq:expect1}, \eqref{eq:stochdom} and Markov's inequality, we have
   \[
   \pr{{S}_1>2 c_1\Psi_d(t)} \leq \frac{1}{2}.
   \]
   We now condition on $\{{S}_1>2c_1\Psi_d(t)\}$.
   After ${X}_1+{\xi}_1$ has spent $2c_1\Psi_d(t)$ time inside
   $B(0,r)$, let $x$ be its position at that time. Then on the event $\{{S}_1>2c_1\Psi_d(t)\}$, we have that ${S}_1 - 2c_1\Psi_d(t)$ is stochastically dominated by $L_x$. Using the fact that~\eqref{eq:expect1} holds for all $x$
   and applying Markov's inequality once more, we obtain that the probability that ${X}_1+{\xi}_1$ spends
   an additional amount of $2c_1\Psi_d(t)$ time inside $B(0,r)$ is again at most $1/2$; that is,
   $$
      \prcond{S_1>4 c_1\Psi_d(t)}{S_1>2 c_1\Psi_d(t)} \leq \frac{1}{2}.
   $$
   Thus, by iterating $n/2$ times, we establish~\eqref{eq:expon}.
   \end{proof}

  \begin{lemma}\label{lem:chernoff}
  For all sufficiently large $K$, there exists a positive constant $c$ so that
   \begin{align}\label{eq:2goal}
   \pr{\sum_{\ell=1}^{4\E{{N}_j}}Y_\ell > \frac{t}{\Psi_d(t)2^{j+1}}} \leq \exp\left(-\frac{c t}{\Psi_d(t)} \right).
   \end{align}
  \end{lemma}

  \begin{proof}[{\bf Proof}]
   
    Let $\theta>0$. Since the random variables $Y_\ell$ are independent, by Chernoff's inequality, we obtain
   \begin{align*}
      &\pr{\sum_{\ell=1}^{4\E{{N}_j}}Y_\ell > \frac{t}{\Psi_d(t)2^{j+1}}}
      \leq \left(\E{e^{\theta Y_1}}\right)^{4\E{{N}_j}} \exp\left(-\frac{\theta t}{\Psi_d(t)2^{j+1}}\right) \\
      & = \exp\left( -\frac{\theta t}{\Psi_d(t)2^{j+1}} + 4\E{{N}_j} \log \E{e^{\theta Y_1}}  \right)
      \leq \exp\left( -\frac{\theta t}{\Psi_d(t)2^{j+1}} + \frac{4c_2\lambda t}{K\Psi_d(t)} \log \E{e^{\theta Y_1}}  \right).
   \end{align*}
   We set $\phi(\theta) = \phi_t(\theta) = \log \E{e^{\theta Y_1}}$, for $\theta \leq \gamma$.
   By the existence of the exponential moment of $Y_1$ (cf.~Lemma~\ref{lem:expon}) and the dominated convergence theorem, we get that $\phi$ is differentiable and its derivative is given by
   \[
   \phi'(\theta) = \frac{\E{Y_1 e^{\theta Y_1}}}{\E{e^{\theta Y_1}}}.
   \]
   Also, $\phi'(0)=\E{Y_1}=0$, and again using the existence of the exponential moment of $Y_1$ and the dominated convergence theorem,
   we have that $\phi'$ is differentiable with derivative given by
   \[
   \phi''(\theta) = \frac{\E{Y_1^2 e^{\theta Y_1}} \E{e^{\theta Y_1}} - \E{Y_1 e^{\theta Y_1}}^2 }{\E{e^{\theta Y_1}}^2}.
   \]
   We will now show that there exists a positive constant $c_2$ such that uniformly over all $t$
   \begin{align}\label{eq:secderiv}
   \phi''(\theta) \leq c_2, \ \text{ for all } \theta\leq \gamma/2.
   \end{align}
   Note that by the definition of $Y_1$ and~\eqref{eq:expect1} and~\eqref{eq:stochdom}, we get 
   $   Y_1 \geq - \frac{\E{S_1}}{\Psi_d(t)} \geq -c_1$.
   Using this when $Y_1 \leq 0$ and the fact that the function $y^2 e^{-\gamma y/4}$ for $y>0$ is maximized at $y=8/\gamma$, we have
   \[
   \E{Y_1^2 e^{\theta Y_1}} \leq c_1^2 + \frac{64}{\gamma^2} \E{e^{(\gamma/4 +\theta)Y_1}}.
   \]
   By Jensen's inequality and the fact that $\E{Y_1}=0$ we obtain
   \[
   \E{e^{\theta Y_1}} \geq \exp(\theta\E{Y_1}) = 1.
   \]
   Thus, Lemma~\ref{lem:expon} and the above two inequalities prove~\eqref{eq:secderiv}.
   
   Since $\phi'(0)=0$ and $\phi''$ is continuous (which follows again by the dominated convergence theorem) we get 
   \[
   |\phi'(\theta)| \leq c_2 \theta, \ \text{ for all } \theta <\gamma/2.
   \]
   Also, since $\phi(0)=0$, we obtain
   \[
   |\phi(\theta)| \leq c_2 \theta^2/2,
   \]
   and hence, we get that there exists $\delta$ small enough such that uniformly for all $t$
   \[
   |\phi(\delta)| \leq 2^{-j-1} \delta.
   \]
   Thus, putting everything together we have 
   \begin{align*}
      &\pr{\sum_{\ell=1}^{4\E{{N}_j}}Y_\ell > \frac{t}{\Psi_d(t)2^{j+1}}}
      \leq \exp\left( -\frac{\delta t}{\Psi_d(t)2^{j+1}} + \frac{4c_2\lambda t}{K\Psi_d(t)}2^{-j-1}\delta \right)\\
      &= \exp\left( -\frac{\delta t}{\Psi_d(t)2^{j+1}}\left(1 - \frac{4c_2\lambda}{K}\right)\right).
   \end{align*}
  Taking now $K$ large enough establishes~\eqref{eq:2goal}.
    \end{proof}

   \begin{proof}[{\bf Proof of Lemma~\ref{lem:iteration}}]

   It only remains to show that there exists a positive constant $c_1$ such that
   \begin{align}\label{eq:finalgoal}
      \pr{\sum_{\ell=1}^{{N}_j}{S}_\ell>\frac{t}{2^j}} \leq \exp(-c_1t/\Psi_d(t)),
   \end{align}
   which together with Lemma~\ref{lem:domination} concludes the proof of Lemma~\ref{lem:iteration}. We can write
   \begin{align*}
      &\pr{\sum_{\ell=1}^{{N}_j}{S}_\ell>\frac{t}{2^j}}
      \leq \pr{\sum_{\ell=1}^{{N}_j}{S}_\ell>\frac{t}{2^j}, {N}_j < 4\E{{N}_j} }+ \pr{{N}_j \geq 4\E{{N}_j}}\\
&\leq \pr{\sum_{\ell=1}^{4\E{{N}_j}}({S}_\ell - \E{S_\ell}) >\frac{t}{2^j} - 4\E{N_j}\E{S_1}} +       \pr{{N}_j \geq 4\E{{N}_j}}\\      
      &\leq \pr{\sum_{\ell=1}^{4\E{{N}_j}}Y_\ell > \frac{t}{2^{j}\Psi_d(t)}\left(1-8\lambda \omega_dr^d/K \right)} + \exp\left(-2 \E{{N}_j}\right),
   \end{align*}
   where the first term on the right-hand side above follows from~\eqref{eq:wald2} and~\eqref{eq:wald} and the fact that $|I_j| \leq \tfrac{t}{2^{j-1}}$. The last term
   follows by applying the Chernoff bound to the Poisson random variable $N_j$.

   If we now choose $K$ large enough we can make $\left(1-8\lambda \omega_d r^d/K \right)$ larger than $1/2$,
   and hence using Lemma~\ref{lem:chernoff} we get the desired tail probability bound, since $\E{{N}_j} = \Theta(t/\Psi_d(t))$ by~\eqref{eq:lambda} and~\eqref{eq:wienervol}.
\end{proof}

\section{Lower bound in $d=1$}\label{sec:lower}
\begin{proof}[{\bf Proof of Theorem~\ref{thm:lower}}]
We want to show that
\[
\pr{\tisol>t} \geq \exp\left(-c\sqrt{t}\log t \log \log t\right),
\]
for some positive constant $c$.
Instead of looking at the interval $[0,t]$, we consider the interval $[t,2t]$ and analyze the event that the origin is
detected throughout $[t,2t]$. Clearly, due to stationarity, this is equivalent to the event $\{\tisol > t\}$.

Now, consider the interval of length $\sqrt{t}$ centered at the origin.
Let $M$ be the set of nodes of $\Pi_0$ that fall in this interval at time $0$. Then, the number of nodes in $M$, which we denote by
$|M|$, is given by a Poisson random variable
with mean $\lambda \sqrt{t}$. Let $C>0$ be a sufficiently large constant that we will set later. We have that
\begin{align*}
   \pr{|M| \geq C\sqrt{t} \log t}
   &\geq \frac{\exp(-\lambda \sqrt{t})(\lambda\sqrt{t})^{C\sqrt{t}\log t}}{(C\sqrt{t}\log t)!} \\
   &\geq \frac{\exp(-\lambda \sqrt{t})(\lambda)^{C\sqrt{t}\log t}}{(C\log t)^{C\sqrt{t}\log t}}
   \geq \exp\left(-c_1 \sqrt{t} \log t \log \log t\right),
\end{align*}
for some positive constant $c_1$.

We now divide the time interval $[t,2t]$ into $t$ subintervals of length $1$. We fix one such subinterval $[s,s+1]$.
The probability that the origin is detected by a given node of $M$ throughout $[s,s+1]$ is at least $\tfrac{c_2}{\sqrt{t}}$ for some positive constant $c_2$.
To see this, note that the probability that this particular node (which started in the interval $[-\sqrt{t}/2,\sqrt{t}/2]$ at time $0$) detects the origin at time $s$ is $\Theta(\tfrac{1}{\sqrt{t}})$ since $s\in [t,2t]$
and, once this node is inside the ball $B(0,r)$ at time $s$, there is a positive probability that it will stay in $B(0,r)$ for one unit of time.

Then, for any given subinterval $[s,s+1]$, we have
\begin{align*}
   &\prcond{\text{no node of $M$ detects the origin throughout $[s,s+1]$}}{|M|\geq C\sqrt{t} \log t} \\
   &\leq \left(1-\frac{c_2}{\sqrt{t}}\right)^{C\sqrt{t} \log t}
   \leq t^{-c_2 C}.
\end{align*}
If for each $s$ there is a node of $M$ detecting the origin throughout $[s,s+1]$, then $\tisol>t$.
Thus, by taking the union bound over all subintervals, we have
\begin{align*}
\prcond{\tisol > t}{|M|\geq C\sqrt{t} \log t}  \geq 1 - t^{-c_2 C+1}.
\end{align*}
Finally,
\begin{align*}
\pr{\tisol>t} &\geq \prcond{\tisol>t}{|M|\geq C\sqrt{t} \log t } \pr{|M| \geq C\sqrt{t} \log t} \\
&\geq (1-t^{-c_2C+1})\exp\left(-c_1 \sqrt{t} \log t \log \log t\right).
\end{align*}
The proof of Theorem~\ref{thm:lower} is then completed by setting $C$ sufficiently large so that $C>1/c_2$.
\end{proof}

\section{Best strategy to avoid isolation}\label{sec:beststrategy}

In this section we prove Theorem~\ref{thm:beststrategy}. The measurability of the event $\{\tisol^D>t\}$ is explained at the end of the section.

In order to prove Theorem~\ref{thm:beststrategy} we first prove a preliminary lemma in the case where time is discrete and 
there is a finite number $k$ of Brownian motions started from uniform points in a big ball. 
Moreover, it will be convenient to generalize the problem so that, instead of having one single collection of sets $(D_s)_s$ for all nodes, we will have one collection of sets for each node. 

\begin{lemma}\label{lem:finitepoints}
Let $(x_i)_{i\leq k}$ be i.i.d.\ uniformly in the ball $B(0,R)$ for some $R>0$ and let $(\xi_i(s))_{i \leq k}$ be independent standard Brownian motions. 
Let $\{U_m^i: m\leq n, i\leq k\}$ be a collection of closed bounded sets in $\R^d$. Then 
\begin{align*}
   &\pr{\forall m=0,\ldots, n, \ \exists i=1,\ldots,k: x_i+\xi_i(m) \in U^i_m } \\
   &\leq \pr{\forall m=0,\ldots, n, \ \exists i=1,\ldots,k: x_i+\xi_i(m) \in B_m^i },
\end{align*}
where $(B_m^i)_{m,i}$ are balls centered at the origin with $\vol{B^i_m} = \vol{U^i_m}$ for all $m$ and $i$. 
\end{lemma}

\begin{proof}[{\bf Proof}]
   We now focus on the first node $x_1+\xi_1$ and define a sequence of stopping times as follows. Let $T_0=0$ and
   \[
   T_1 = \inf\{m\geq 0: \forall i=2,\ldots,k\ x_i+\xi_i(m) \notin U_m^i\}.
   \]
   Define inductively
   \[
   T_{j+1} = \inf\{m\geq T_j+1: \forall i=2,\ldots, k \ x_i+\xi_i(m) \notin U_m^i\}.
   \]
   Let $\kappa = \sup\{\ell\geq 0: T_\ell \leq n\}$. Then we have
   \[
   \pr{\forall m=0,\ldots, n, \ \exists i =1,\ldots k: x_i+ \xi_i(m) \in U_m^i} = \E{\prod_{j=1}^{\kappa}\1(x_1+\xi_1(T_j) \in U_{T_j}^1)}.
   \]
   By the independence of the motions of the nodes $1,\ldots, k$ and the Markov property, the right-hand side above can be written as 
   \[
   \E{\int_{\R^d}\cdots \int_{\R^d}\frac{\1(z_0 \in B(0,R))}{\vol{B(0,R)}} \prod_{j=1}^{\kappa} \1(z_j \in U_{T_j}^1) p_{T_j -T_{j-1}}(z_{j-1},z_j)\,dz_0\ldots dz_\kappa },
   \]
   where $p_t(x,y)$ stands for the transition kernel of Brownian motion. 
   Applying the rearrangement inequality as in~\cite[Theorem~1.2]{BrascLiebLutt} to the integral appearing inside the expectation (the transition kernel 
   $p_t(x,y)$ of the Brownian motion is symmetric decreasing as a function of the distance $|x-y|$), 
   we get that this last expression is smaller than
   \[
   \E{\int_{\R^d}\cdots \int_{\R^d}\frac{\1(z_0 \in B(0,R))}{\vol{B(0,R)}} \prod_{j=1}^{\kappa} \1(z_j \in B_{T_j}^1)p_{T_j -T_{j-1}}(z_{j-1},z_j)\,dz_0\ldots dz_\kappa},
   \]
   which is equal to
   \[
   \pr{\forall m=0,\ldots, n, \ \exists i =1,\ldots ,k: x_i+ \xi_i(m) \in 
   V_m^i},
   \]
   where $V_m^i = U_m^i$ for $i=2,\ldots, k$ and $V_m^1 = B_m^1$ for all $m$. 
   
   Continuing in the same way, i.e.\ fixing node $2$ and looking at the times that the other particles, $1,3,4,
   \ldots, k$ do not detect the target before time $n$, we get that this last probability is increased when the sets $V_m^2$
   are replaced by the balls $B_m^2$ for all $m$. Then we apply the same procedure for nodes $3,4,\ldots,k$ and this concludes the proof.
\end{proof}

Before proving Theorem~\ref{thm:beststrategy}, we give some definitions that will be used in the proofs repeatedly.

For $n \in \N$ and $t>0$, define the dyadic rationals of level $n$ as
\[
\mathcal{D}_{n,t}  = \left\{\frac{jt}{2^n}: j=0,\ldots,2^n \right\}.
\]
Let $(U_s)_{s\leq t}$ be closed sets in $\R^d$. 
For each $s$ and $n$, we define the set
\begin{align}\label{eq:defsets}
U_{s,n} = \{z \in \R^d: d(z,U_s) \leq \left(t/2^n\right)^{1/3}\},
\end{align}
which is clearly closed. (The metric $d(x,A)$ stands for the Euclidean distance between the point $x$ and the set $A$.)

For every $\ell \in \mathcal{D}_n$ 
(we drop the dependence on $t$ from $\D_{n,t}$ to simplify notation), we will define a set $\til{U}_{\ell,n}$ as follows. For each such $\ell$ take $s = s(\ell) \in [\ell,\ell+t/2^n)$ such that 
\[
\vol{U_{s,n}} \leq \inf_{u \in [\ell,\ell+t/2^n)} \vol{U_{u,n}} +\frac{1}{n}.
\]
We now define $\widetilde{U}_{\ell,n} = U_{s(\ell),n}$
and 
finally for every $n$ and $i$ we let
\begin{align}\label{eq:defomega}
&\Omega_{n,i} = \left\{ \forall h \leq t/2^n: \sup_{s,u: |s - u|\leq h} \|\xi_i(s) - \xi_i(u) \| \leq \frac{h^{1/3}}{2}\right \}.
\end{align}

\begin{lemma}\label{lem:discreteapprox}
Let $(U_s)_{s\leq t}$ be closed sets in $\R^d$ that are uniformly bounded; i.e., there exists $L_t>0$ such that $\cup_{s\leq t}U_s \subseteq B(0,L_t)$. Then, with the definitions given above, we have that, almost surely,
\begin{align*}
\{\forall s\leq t, \ \exists i: X_i+\xi_i(s) \in U_s\} \subseteq \bigcup_{n_0}\bigcap_{n\geq n_0}\{\forall \ell \in \D_n, \ \exists i: X_i + \xi_i(\ell) \in \til{U}_{\ell,n} \}.
\end{align*}
\end{lemma}

\begin{proof}[{\bf Proof}]
We first notice that, almost surely,
\begin{align}\label{eq:limitpoisson}
\{\forall s\leq t, \ \exists i: X_i+\xi_i(s) \in U_{s}\} =\cup_{R} \{\forall s\leq t, \ \exists i =1,\ldots N_R: X_i + \xi_i(s) \in U_s\},
\end{align}
where $N_R$ is the number of nodes of the Poisson process that started in the ball $B(0, R)$, so $N_R$ is a Poisson random variable of parameter $\lambda \vol{B(0,R)}$. Indeed,
if $F_n$ denotes the event that some node that started outside the ball $B(0,n)$ detects the target before time $t$, then we will show that 
\[
\pr{F_n} \to 0 \ \text{ as } n\to \infty.
\]
Let $\Phi_n$ be the point process defined as follows
\[
\Phi_n = \{X_i \in \Pi_0: X_ i \notin B(0,n) \text{ and } \exists s \leq t: X_i +\xi_i(s) \in U_s \}.
\]
Then, by the thinning property of Poisson processes, $\Phi_n$ is a Poisson process of total intensity
\[
\E{\Phi_n(\R^d)} = \lambda \E{\vol{\cup_{s\leq t} \left( \xi(s) + U_s\right) \cap B(0,n)^c  }},
\]
where $(\xi(s))_s$ is a Brownian motion starting from the origin.
Clearly, by Markov's inequality, we have
\[
\pr{F_n} = \pr{\Phi_n(\R^d) \geq 1} \leq \E{\Phi_n(\R^d)}.
\]
Since for all $s\leq t$ the sets $U_s$ are contained in $B(0,L_t)$, we have
\[
\cup_{s\leq t} \left( \xi(s) + U_s\right) \subseteq \cup_{s\leq t} \left( \xi(s) + B(0,L_t)\right).
\]
As $n\to \infty$, by dominated convergence, we have that 
\[
\E{\vol{\cup_{s\leq t} \left( \xi(s) + U_s\right) \cap B(0,n)^c  }} \to 0,
\]
since $\vol{\cup_{s\leq t} \left( \xi(s) + U_s\right) \cap B(0,n)^c  } \leq \vol{\cup_{s\leq t} \left( \xi(s) + B(0,L_t)\right)}$ and the latter has finite expectation given by~\eqref{eq:wienervol} for $r=L_t$. 

We will now show that, on the event $\cap_i \cup_n \Omega_{n,i}$, the following holds for all $k$:
\begin{align}\label{eq:goalmeasurability}
\{\forall s\leq t, \ \exists i \leq k: X_i + \xi_i(s) \in U_s\} \subseteq \bigcup_{n_0}\bigcap_{n\geq n_0}\{\forall \ell \in \D_n, \ \exists i\leq  k: X_i + \xi_i(\ell) \in \til{U}_{\ell,n} \}.
\end{align}
Take $n_0$ large enough so that $\Omega_{n,i}$ holds for all $n\geq n_0$ and all $i=1,\ldots, k$ (since the sets $\Omega_{n,i}$ are increasing in $n$). 
We want to show that, for all $\ell \in \D_n$, there exists $i=1,\ldots,k$ for which $X_i + \xi_i(\ell) \in \til{U}_{\ell,n}$. Take $i$ such that $X_i + \xi_i(s(\ell)) \in  U_{s(\ell)}$. Then we have
\begin{align*}
d(X_i + \xi_i(\ell), U_{s(\ell)}) \leq d(\xi_i(\ell),\xi_i(s(\ell))) + d(X_i + \xi_i(s(\ell)), U_{s(\ell)}) \leq \frac{1}{2} (t/2^n)^{1/3} < (t/2^n)^{1/3},
\end{align*}
since $U_{s(\ell)}$ is a closed set.

%

By the same reasoning that led to~\eqref{eq:limitpoisson} we get that, almost surely,
\[
\cup_R \cup_{n_0} \cap_{n\geq n_0}\{\forall \ell \in \D_n, \ \exists i=1,\ldots, N_R: X_i + \xi_i(\ell) \in \til{U}_{\ell,n} \} 
= \cup_{n_0} \cap_{n\geq n_0}\{\forall \ell \in \D_n, \ \exists i: X_i + \xi_i(\ell) \in \til{U}_{\ell,n} \}.
\]
This together with the fact that $\pr{\cap_i\cup_n \Omega_{n,i}} =1$, which follows from L\'evy's modulus of continuity theorem (see for instance \cite[Theorem~1.14]{BM}), concludes the proof of the lemma.
\end{proof}

\begin{proof}[{\bf Proof of Theorem~\ref{thm:beststrategy}}]

By Lemma~\ref{lem:discreteapprox} we have that
\begin{align*}
\pr{\forall s\leq t, \ \exists i: X_i+\xi_i(s) \in D_s} \leq \lim_{n_0\to \infty} \lim_{n_1\to \infty}\pr{\bigcap_{n=n_0}^{n_1} \{\forall \ell \in \D_n, \ \exists i: X_i+\xi_i(\ell) \in \til{D}_{\ell,n} \} }.
\end{align*}
Since the sets $(D_s)$ are uniformly bounded, by the same reasoning that led to~\eqref{eq:limitpoisson} we get
\begin{align*}
\bigcap_{n=n_0}^{n_1} \{\forall \ell \in \D_n, \ \exists i: X_i+\xi_i(\ell) \in \til{D}_{\ell,n} \}  
= \bigcup_{R} \bigcap_{n=n_0}^{n_1} \{\forall \ell \in \D_n, \ \exists X_i \in \Pi_0 \cap B(0,R): X_i+\xi_i(\ell) \in \til{D}_{\ell,n}\}. 
\end{align*}
Let $N_R$ be the number of nodes of the Poisson process $\Pi_0$ that are in $B(0,R)$. Then $N_R$ is a Poisson random variable of parameter $\alpha=\lambda \vol{B(0,R)}$. 
If we condition on $N_R$, then by standard properties of Poisson processes, we get that the positions of the nodes $X_i$ are independent and uniformly distributed in $B(0,R)$. 
So we obtain 
\begin{align*}
&\pr{\bigcap_{n=n_0}^{n_1} \{\forall \ell \in \D_n, \ \exists X_i \in \Pi_0 \cap B(0,R): X_i+\xi_i(\ell) \in \til{D}_{\ell,n}\}} 
\\
&=
\sum_{k = 0}^{\infty} e^{-\alpha} \frac{\alpha^k}{k!} \pr{\bigcap_{n=n_0}^{n_1} \{\forall \ell \in \D_n, \ \exists i=1,\ldots,k : x_i+\xi_i(\ell) \in \til{D}_{\ell,n} \} },
\end{align*}
where the $x_i$'s are i.i.d.\ uniformly in the ball $B(0,R)$.

Using Lemma~\ref{lem:finitepoints} we deduce that, for all $k$,
\begin{align*}
   \pr{\bigcap_{n=n_0}^{n_1} \{\forall \ell \in \D_n, \ \exists i\leq k : x_i+\xi_i(\ell) \in \til{D}_{\ell,n} \} } 
   \leq \pr{\bigcap_{n=n_0}^{n_1} \{\forall \ell \in \D_n, \ \exists i\leq k : x_i+\xi_i(\ell) \in B(0,\til{r}_{\ell,n}) \} },
\end{align*}
where $\til{r}_{\ell,n}$ satisfies $\vol{B(0,\til{r}_{\ell,n})} = \vol{\til{D}_{\ell,n}}$. 
\newline
Thus if $r_{s,n}$ is such that $\vol{B(0,r_{s,n})} = \vol{D_{s,n}}$, then for every $s \in [\ell,\ell+t/2^n)$
\begin{align}\label{eq:rlnineq}
\til{r}_{\ell,n} \leq  r_{s,n}\left( 1 + \frac{1}{r_{s,n}^d c(d)n}\right)^{1/d},
\end{align}
where $c(d)$ is a constant that depends only on the dimension.

Hence we get 
\begin{align*}
\pr{\bigcap_{n=n_0}^{n_1} \{\forall \ell \in \D_n, \ \exists i: X_i+\xi_i(\ell) \in \til{D}_{\ell,n} \} } \leq \pr{\bigcap_{n=n_0}^{n_1} \{\forall \ell \in \D_n, \ \exists i: X_i+\xi_i(\ell) \in B(0,\til{r}_{\ell,n}) \} }
\end{align*}
and thus
\begin{align*}
\pr{\forall s\leq t, \ \exists i: X_i+\xi_i(s) \in D_s} \leq \pr{\bigcup_{n_0} \bigcap_{n\geq n_0} \{\forall \ell \in \D_n, \ \exists i: X_i +\xi_i(\ell) \in B(0,\til{r}_{\ell,n}) \}}.
\end{align*}
Now it only remains to show that a.s.\
\begin{align}\label{eq:goalballs}
 \{ \forall s\leq t, \ \exists i: X_i+\xi_i(s) \in B_s\} = \bigcup_{n_0} \bigcap_{n\geq n_0} \{\forall \ell \in \D_n, \ \exists i: X_i +\xi_i(\ell) \in B(0,\til{r}_{\ell,n}) \}.
\end{align}
In the notation introduced before Lemma~\ref{lem:discreteapprox} we have $B(0,\til{r}_{\ell,n}) = \til{B}_{\ell,n}$. 
Then applying Lemma~\ref{lem:discreteapprox}, we get that the left-hand side of~\eqref{eq:goalballs} is contained in the right-hand side of~\eqref{eq:goalballs}. 

To show the other inclusion, notice first that since all the balls are uniformly bounded, by the same reasoning that led to~\eqref{eq:limitpoisson}, 
it suffices to look at a finite number of nodes of the Poisson process and show that a.s.\
\begin{align}\label{eq:goalballs2}
\bigcup_{n_0} \bigcap_{n\geq n_0} \{\forall \ell \in \D_n, \ \exists i\leq k: X_i +\xi_i(\ell) \in B(0,\til{r}_{\ell,n}) \} \subseteq \{ \forall s\leq t,\  \exists i\leq k: X_i+\xi_i(s) \in B_s\}.
\end{align}

In order to establish~\eqref{eq:goalballs2}, 
notice first that the events $\Omega_{n,i}$ are increasing in $n$ and thus, almost surely, there exists $n_0$ large enough so that $\Omega_{n,i}$ holds for all $n\geq n_0$ and all $i=1,\ldots, k$. 
If $\ell$ is such that $\ell \leq s< \ell+t/2^n$,
then there exists $i=1,\ldots,k$ such that $X_i + \xi_i(\ell) \in B(0,\til{r}_{\ell,n})$, and hence using~\eqref{eq:rlnineq} we get 
\[
d(X_i + \xi_i(\ell), B_s) = (\|X_i + \xi_i(\ell)\|_2 - r_s)^+ \leq r_{s,n}\left( 1 + \frac{1}{r_{s,n}^d c(d)n}\right)^{1/d}  - r_s.
\]
Therefore, for all $n\geq n_0$, by the triangle inequality again we have
\begin{align*}
\min_{i=1,\ldots,k} d(X_i + \xi_i(s),B_s)& \leq \min_{i=1,\ldots,k} ( d(\xi_i(s),\xi_i(\ell)) + d(X_i + \xi_i(\ell),B_s) ) \\
&\leq \frac{1}{2}(t/2^n)^{1/3} + r_{s,n}\left( 1 + \frac{1}{r_{s,n}^d c(d)n}\right)^{1/d}  - r_s \to 0
\end{align*}
as $n\to \infty$, since $r_{s,n} \to r_s$ as $n\to \infty$.
Hence, this gives that there exists $i \in \{1,\ldots, k\}$ such that $X_i + \xi_i(s) \in B_s$, since $B_s$ is a closed set and  this finishes the proof of~\eqref{eq:goalballs2} and concludes the proof of the theorem.
\end{proof}


We now explain the measurability issue raised at the beginning of the section. 

\begin{lemma}
Let $(D_s)_s$ be a collection of closed sets in $\R^d$ that are uniformly bounded; i.e., there exists $L_t>0$ such that $\cup_{s\leq t}D_s \subseteq B(0,L_t)$. Then, for all $t\geq 0$, the event 
$\{\tisol^D >t\}$ is measurable. 
\end{lemma}

\begin{proof}[{\bf Proof}]
By the assumption on the sets being uniformly bounded, as in~\eqref{eq:limitpoisson} we can write 
\[
\{\forall s\leq t, \ \exists i: X_i + \xi_i(s) \in D_s\}  = \cup_R \{ \forall s\leq t,\ \exists X_i \in \Pi_0 \cap B(0,R): X_i + \xi_i(s) \in D_s\}.
\]
In order to show the measurability, it suffices to show that, for all $k$, the event
\[
\{ \forall s\leq t, \ \exists i=1,\ldots, k: X_i+\xi_i(s) \in D_s\}
\]
is measurable. But the event above can be alternatively written as
\[
\{ \forall s\leq t, \ (X_1 + \xi_1(s),\ldots, X_k+\xi_k(s)) \in D_s^{\otimes k} \},
\]
where $D_s^{\otimes k} = \{ x=(x_1,\ldots,x_k) \in \R^{dk}: \ \exists i \text{ s.t. } x_i \in D_s\}$ is clearly a closed set. So the initial question of measurability reduces to the question of measurability of the event 
\[
\{\forall s\leq t, \ \xi(s) \in U_s\},
\]
where $\xi$ is a Brownian motion in $dk$ dimensions and $(U_s)$ is a collection of closed sets. 
In order to show this, we use the same notation as in~\eqref{eq:defsets} and  
define $\Omega_n$ as in~\eqref{eq:defomega} but only for one Brownian motion and $Z_{\ell} = \bigcap_{\ell \leq s < \ell + t/2^n} D_{s,n}$,
which is again closed as an intersection of closed sets. Then using similar ideas as in the proof of~\eqref{eq:goalmeasurability} and~\eqref{eq:goalballs2} we get that 
on $\cup_n \Omega_n$ 
\begin{align*}
\{\forall s\leq t, \xi(s) \in D_s\} = \bigcup_{n_0}\bigcap_{n\geq n_0}\{\forall \ell \in \D_n,  \xi(\ell) \in Z_\ell \}.
\end{align*}
Hence  the measurability follows, since by L\'evy's modulus of continuity theorem (see for instance \cite[Theorem~1.14]{BM})  we have that $\pr{\cup_m \Omega_{m}} =1$.
\end{proof}

\section{Concluding remarks and questions}\label{sec:questions}

A related quantity that has been studied for Poisson Brownian motions is the \emph{detection time}, $\tdet$. Consider a
target particle $u$ and define $\tdet$ as the first time at which a node of the Poisson point process is within distance $r$ of $u$.
Kesidis, Konstantopoulos and Phoha~\cite{Kesidis,Konst} used a result from stochastic geometry~\cite{SKM95} to show that,
when $u$ stays fixed at the origin,
\begin{equation}
   \pr{\tdet > t} = \exp\left(-\lambda\E{\vol{W_\origin(t)}}\right)
   = \exp\left(-c_d\frac{t}{\Psi_d(t)}(1+o(1))\right),
   \label{eq:wiener}
\end{equation}
where $W_\origin(t)$ is the Wiener sausage as defined in~\eqref{eq:saucisse} and $c_d$ is an explicit constant.

Even though the isolation time seems to be similar to the detection time, we are not aware of any reduction that allows us to use ideas from
stochastic geometry to characterize $\tisol$.

{\bf{Question.}} Does the tail of $\tisol$ behave similarly to the tail of $\tdet$? Namely, is it true that for all dimensions $d\geq 1$, there exists a constant $\tilde c_d$ such that
   \begin{equation}
      \pr{\tisol > t} = \exp\left(-\tilde c_d\frac{t}{\Psi_d(t)}(1+o(1))\right)?
      \label{eq:q1}
   \end{equation}

Peres et al.~\cite{PSSS11} and Peres and Sousi~\cite{PS11} studied the detection time for the case when $u$ also moves. Among other things, they established that, when the target is allowed to move 
independently of the nodes of $\Pi_0$, then the best strategy for $u$ to avoid detection
is to stay fixed and not to move.
Similar results were obtained for random walks in the lattice $\zset^d$ by Moreau et al.~\cite{MOBC} and Drewitz et al.~\cite{DGRS}.
It is interesting that staying fixed is also the best strategy to avoid isolation, cf. Corollary~\ref{cor:strategy}.


We now discuss some additional motivation and conclude with another open problem. 
For each $s\geq 0$, let $G_s$ denote the graph with vertex set $\Pi_s$ and an edge between any two nodes of $\Pi_s$ that are within distance $r$ of each other.
As in~\cite{PSSS11}, we call this stationary sequence of graphs the \emph{mobile geometric graph}.
This and other variants have been considered as models for mobile wireless networks,
which motivated the study of some properties of this types of graphs, 
such as broadcast~\cite{clementi,PPPU10,PSSS11,LLMSW12}, spread of infection~\cite{KestenSidoravicius,KS04}, 
detection of targets~\cite{Kesidis,Konst,PSSS11,PS11,S11} and percolation~\cite{SS10,PSSS11}. We refer the reader to the discussion in~\cite{SS10} for additional motivation and related work in 
the engineering literature.

Regarding percolation properties of $G_s$, 
it is known~\cite{vandenberg} that, for $d\geq 2$, there exists a constant $\lambda_\critical=\lambda_\critical(d)$ such that,
if $\lambda>\lambda_\critical$, then $G_s$ contains an infinite connected component at all times.
Peres et al.~\cite{PSSS11} considered the regime $\lambda>\lambda_\critical$ and derived lower and upper bounds for the so-called
\emph{percolation time}, which is the first time $\tperc$ at which a non-mobile target $u$ belongs to the infinite connected component.
A quantity related to the isolation time is the \emph{non-percolation} time $\tnonperc$, which is the first time at which $u$ does \emph{not} belong
to the infinite connected component. Clearly $\tnonperc \leq \tisol$. We conclude with the question below.

{\bf{Question.}} Do the 
tail probabilities of $\tperc$ and $\tnonperc$ satisfy~\eqref{eq:q1}?

\section*{Acknowledgments}
We are grateful to Almut Burchard for useful discussions on rearrangement inequalities. 
We thank Philippe Charmoy and the referee for helpful comments. Part of this work was done while the second author was visiting Microsoft Research.

\bibliographystyle{plain}
\bibliography{isol}

\end{document}